\documentclass[12pt]{amsart} \usepackage{amssymb}
\usepackage[colorlinks=true, pdfstartview=FitV, linkcolor=blue, citecolor=blue, urlcolor=blue]{hyperref}
\usepackage{amscd} \usepackage{longtable}
\calclayout

\numberwithin{equation}{subsection}

\newtheorem*{theorem*}{Theorem}
\newtheorem*{corollary*}{Corollary}
\newtheorem{theorem}[subsection]{Theorem}
\newtheorem{lemma}[subsection]{Lemma}
\newtheorem{proposition}[subsection]{Proposition}
\newtheorem{corollary}[subsection]{Corollary}

\theoremstyle{definition}

\theoremstyle{remark} \newtheorem{remark}[subsection]{Remark}
\newtheorem{remarks}[subsection]{Remarks}

\def\im{\mathrm{i}}
\def\inv{^{-1}}

\def\today{\ifcase\month\or January\or
  February\or March\or April\or May\or June\or July\or August\or
  September\or October\or November\or December\fi \space\number\day,
  \number\year} \def\C{\mathbb C} 
\def\quot#1#2{{#1/\!\!/#2}}
\def\twist#1#2#3{#1\times^{#2}#3}
\def\c{^\C} 
\def\lie#1{\mathfrak{ #1}}

\def\liek{\lie k}
\def\liep{\lie p}

\def\liea{\lie a}
\def\lieg{\lie g}
\def\liem{\lie m}
\def\lieu{\lie u}

\def\R{{\mathbb R}}

\newcommand{\Ad}{\operatorname{Ad}}
\def\mup{\mu_{\liep}}
\def\mua{\mu_{\liea}}

\def\mup{\mu_{\liep}}

\def\Mup{\mathcal M_{\liep}}

\def\max{\operatorname{max}}
\def\grad{\operatorname{grad}}
\title{A structure theorem along fibers of extreme points of the momentum polytope}
\author{Peter Heinzner}
\address{Fakult\"at f\"ur Mathematik\\
  Ruhr Universit\"at Bochum\\
  Universit\"atsstrasse 150\\
  D - 44780 Bochum}
  \email{peter.heinzner@rub.de}
  \author{Christian Z\"oller}
  \address{Fakult\"at f\"ur Mathematik\\
  Ruhr Universit\"at Bochum\\
  Universit\"atsstrasse 150\\
}
\email{christian.zoeller@rub.de}

\begin{document}

\begin{abstract}
Let $G$ be a complex reductive Lie group
acting on a compact K\"ahler manifold $X$ and assume that the action
of a maximal compact subgroup $K$ of $G$ is Hamiltonian. For each extreme point
of the convex hull of the momentum map image, there is an associated
open dense subset of $X$, which is invariant under a parabolic subgroup
$Q$ of $G$. We prove a $Q$-equivariant product decomposition for the
$Q$-action on this subset and discuss some applications of the result.
We show a similar statement for real reductive subgroups of $G$ for the
restricted momentum map.
\end{abstract}
\maketitle

\begin{center}
\emph{In memory of Joseph A.Wolf}
       \end{center}

\section*{Introduction}

For a finite dimensional $G$-representation $V$ and $X=\mathbb{P}(V)$ the following local structure theorem
 holds for any point $x\in X$ with compact $G$-orbit (see \cite{BLV} for the precise statement or Section \ref{section:projective}). There is a Zariski open neighborhood $\Omega$ of $x$, a parabolic subgroup $Q$ of $G$ with Levi decomposition $Q=R\rtimes L$, an $L$-representation $F$ such that the
 bundle  $\twist{Q}{L}{F}$ associated to the $L$-principal bundle $Q\to Q/L$ is   $Q$-equivariantly
 and algebraically isomorphic to $\Omega$. The $Q$ orbit through $x$ is open in the compact orbit.
 Note that the map $R\times F\to \twist{Q}{L}{F}$ induced by
the inclusion $R\times F\subset Q\times F$ is an isomorphism. In particular $\Omega$ is an affine neighbourhood of $x$.

In this paper we prove a version of the local structure theorem for an invariant compact submanifold of a K\"ahler manifold.
More precisely, we consider a holomorphic action of a connected complex
reductive group $U^\C$ with maximal compact subgroup $U$ on a K\"ahler manifold $Z$ such that the restriction $U$
acts on $Z$ in a Hamiltonian fashion. This means, that
the K\"ahler form $\omega$ on $Z$ is $U$-invariant and there is a $U$-equivariant momentum
map $\mu\colon Z\to \lie u^*$. We will use the notations introduced in \cite{HSS},
i.e., we fix a $U$-invariant inner product $<\, ,\, >$ on $\lieu$ and use it
to identify $\lieu$ with its dual. We also identify $\lieu$  with $i\lieu$. With
this  identifications in mind a momentum map $\mu$ on $Z$
is a smooth $U$-equivariant map $\mu\colon Z\to \im \lie u$ such that for every
function $\mu^\alpha\colon Z\to \R$, $\mu^\alpha(x)=<\mu(x),\alpha>$ where $\alpha\in\im \lieu$
the  gradient condition
\[
 \grad \mu^\alpha=\alpha_Z
\]
is satisfied. The gradient is computed with respect to the Riemannian structure given by the K\"ahler form $\omega$ and  $\alpha_Z$ denotes the vector field on $Z$ given by the one-parameter subgroup $t\mapsto \exp t\alpha$ of $U^\C$ and the $U^\C$-action on $Z$.  Note that
for every $\xi\in\lieu$ the function $\mu^{i\xi}$ is
a Hamiltonian function with Hamiltonian vector field $\xi_Z$ with respect to the symplectic form $\omega$.

In the following we also fix a closed real subgroup $G$ of $U^\C$ which is compatible with the Cartan decomposition
$U^\C=U\exp(\im \lieu)$ in the sense that for $K:=U\cap G$ and $\liep:=\lieg\cap \im\lieu$ the mapping

\[
K\times \liep\to G, (k,\alpha)\mapsto k\exp\alpha
\]
is an isomorphism.

The
gradient condition does not change, if $Z$ is replaced by a $G$-invariant Riemannian submanifold of $Z$. It only
depends on the $G$-orbit.
We fix a $G$-invariant compact submanifold $X$ of $Z$.  On $X$ we have the restricted momentum map or
gradient map
\[\mup\colon X\to \liep\]
which is defined by $\mup^\alpha=\mu^\alpha\vert X$ for $\alpha \in \liep$. Note that for $X=Z$ the restricted
momentum map $\mup$ is just the map $\pi_\liep\circ\mu$ where $\pi_\liep\colon \im\lieu\to \liep$  denotes the orthogonal projection.

Let $\liea$ be a maximal Lie subalgebra of $\lieg$ which lies in $\liep$.  The map $
\mu_\liea:=\pi_\liea\circ \mu$ where $\pi_\liea\colon \liep \to \liea$
denotes the orthogonal projection of $\liep$ onto $\liea$ is a restricted momentum map for the $A=\exp\liea$-action
on $X$. The convex envelope $P$ of $\mu_\liea(X)$ is a convex polytope.

For simplicity  we refer to the case where $G$ is complex reductive and $X$ is a compact complex submanifold of $Z$
as the complex case. In the complex case we have $\liep=\im \liek$ and in this case we also have $P=\mu_\liea(X)$
for a connected manifold $X$ (\cite{atiyah,GS}).
In fact $P$ is the convex envelope  of the finite subset $\mu_a(X^A)$ of $\liea$ where $X^A$ denotes the set of $A$-fixed points in $X$. In general we do not know any connected smooth compact example where $\mu_\liea(X)$ is not convex.

We show in section \ref{section:proof} a version of the local structure theorem  at any
$x\in\mu_\liea^{-1}(\sigma)$ where $\sigma$ is an extreme point of
the momentum polytope $P$ or equivalently for any extreme point of the convex envelope $E$ of the image $\mup(X)$.

For a precise formulation of the main result we have to introduce more notation. Let $E$ be a
convex body in a finite dimensional Euclidian vector space $V$ with inner product $<\, ,\, >$.
Any $\beta\in V$ defines a face $F_\beta(E):= \{\alpha\in E: <\beta, \alpha>=\max\{<\beta,\gamma>: \gamma\in E\}\}$. A face $F$ of $E$ is called exposed by $\beta$ if $F=F_\beta(E)$.
In this case we also say that $\beta$ exposes $F$. If $F$ is exposed then the  set $C_F$ of
$\beta$ which exposes $F$ is a cone in $V$.

Let $E$ be the convex envelope of $\mup(X)$. It is known that every
face $F$ of $E$ is exposed (see \cite{BGH2}).
Now let $\sigma$ be an extreme point of $E$. The cone $C_\sigma$ is invariant under the
group $K^\sigma:=\{g\in K: g\cdot \sigma=\sigma\}$ which is the centralizer of $\sigma$ in $K$. Since $C_\sigma$ is $K^\sigma$-invariant and convex, the set $C_\sigma^{K^\sigma}$ of $K^\sigma$-fixed points in $C_\sigma$ is non empty and $\sigma$ is exposed by some $\beta\in \liep^{K^\sigma}$.

Any $\beta\in\liep$ defines the  parabolic subgroups
$G^{\beta_{-}}:=\{  g\in G: \lim\limits_{t \to +\infty} \exp(t\beta) g   \exp(-t\beta) \text{  exists}\}$ and $G^{\beta_{+}}:=\{  g\in G: \lim\limits_{t \to -\infty} \exp(t\beta) g   \exp(-t\beta) \text{  exists}\}$.
The Levi factor $G^{\beta_{-}}\cap G^{\beta_{+}}$
is the centralizer $G^{\beta}$ of $\beta$ in $G$ and the unipotent radical
of $G^{\beta_-}$ is given by
$R^{\beta_{-}}:=\{  g\in G : \lim\limits_{t \to +\infty} \exp(t\beta) g
\exp(-t\beta)=e\}$.

We also have the subsets
$X_{\max}^\beta=\{x\in X: \mup^\beta(x)=\max\mup^\beta(X)\}\subset X^\beta:=\{x\in X:\beta_X(x)=0\}$
and
$X_{\max}^{\beta_-}=\{x\in X: \lim\limits_{t \to +\infty} \exp(t\beta)\cdot x \text{ exists in } X_{\max}^{\beta_-}\}$
of $X$ which are $G^{\beta}$-stable and we have a surjective smooth $G^\beta$-equivariant map
$p^{\beta_{-}}\colon X_{\max}^{\beta_{-}}\to X^{\beta}_{\max}$, $p^{\beta_{-}}
(z)=\lim\limits_{t \to +\infty}\exp(t\beta)\cdot z$ (see section \ref{section:parabolic} for more details).
The set $X_{\max}^{\beta_-}$ is open and $G^{\beta_-}$-stable, $X_{ \max}^\beta$ is a smooth $G^\beta$-stable
compact submanifold of $X$ and
$p^{\beta_{-}}\colon X_{\max}^{\beta_{-}}\to X^{\beta}_{\max}$ is a smooth $G^\beta$-equivariant fibration
(see \cite{HSS} and section \ref{section:parabolic} for more details).

Our main result is the following

 \begin{theorem*} Let $X$ be a $G$-invariant compact submanifold of $Z$ and let $\sigma$ be an extreme point of the convex envelope $E$ of $\mup(X)$. Then for any
$\beta\in C_{\sigma}^{K^\sigma}$ there is a normal compatible subgroup $I^\beta$ of $G^\beta$ containing $\exp\R\beta$ such that $G^\beta/I^\beta$ is compact and for every $x\in X_{\max}^{\beta}$ the following holds.

There is an open neighbourhood $U \subseteq  X^{\beta}_{max} $ of $x$ , an $I^{\beta}$-representation $F$, with all weights of $\beta$ being strictly negative and a $R^{\beta_-}\rtimes I^\beta$-equivariant smooth isomorphism
	\[
	\Phi:   R^{\beta_{-}} \times F \times U \to  (p^{\beta_-})^{-1}(U)
	\]
such that $p^{\beta_-}\circ \Phi=p_U$, where $p_U$ is the projection onto $U$.
 \end{theorem*}

For $x\in\mup\inv (\sigma) $ the orbits $G^\beta\cdot x$ and $G\cdot x$ are compact. It should also
be underlined that $\mup^{-1}(\sigma)=X^{\beta}_{\max}$ for every $\beta$ in $C_ \sigma$. Moreover for $\beta\in C_{\sigma}^{K^\sigma}$ the open subset  $X^{\beta_-}_{\max}$ and the smooth fibration $p^{\beta_-}$ only depend on the face $\sigma$ and not the choice of $\beta\in C_{\sigma}^{K^\sigma}$  (see  \cite{BGH2} for more details).

 In section \ref{section:quotients} we apply the above theorem  in order to obtain the following.

 \begin{theorem*}
  Let $\sigma$ be an extreme point of the convex envelope $E$ of ${\mup(X)}$ and $\beta\in C_\sigma$. Then
  $\mu^{-1}(\sigma)=X^{\beta}_{\max}\subset X^{G^\beta}$. If $\beta$ is $K^\sigma$-invariant, then $K^\beta=K^\sigma$ and the $R^{\beta_-}$-action on $X_{\max}^{\beta_-}$ is
  proper and free. The quotient
  $X_{\max}^{\beta_-}/R^{\beta_-}$ is a complex manifold with an induced holomorphic
  $G^\beta$-action. The  map $p^{\beta_-}:X^{\beta_-}_{\max}\to X_{\max}^\beta$ induces a $I^\beta$-invariant holomorphic bundle map $q\colon X_{\max}^{\beta_-}/R^{\beta_-}\to X^{\beta}_{\max}$
  with typical fiber $F$.
 \end{theorem*}

The same theorem holds in the complex case. In this case $X^\beta_{\max}$ is connected if $X$ is connected and $I^\beta$ contains the connected component of the identity of $G^\beta$. If
$G$ is connected then we have $G^\beta=I^\beta$. In particular
the holomorphic bundle map $q\colon X_{\max}^{\beta_-}/R^{\beta_-}\to X^{\beta}_{\max}$ is a $G^\beta$-invariant
Stein map (see \cite{HMP}). The same is true
for the holomorphic map $p^{\beta_-}\colon X^{\beta_-}_{\max}\to X^\beta_{\max}$.

 Let $G$ be a connected complex reductive group. For an extreme point $\sigma$ of $E$  and every
 $\beta\in C_\sigma$ we have $X^\beta_{\max}\subset X^{G^\beta}$ and for
 $\beta\in C_\sigma^{K^\sigma}$ we obtain the following
 complex version of the above theorem.

 \begin{theorem*}
For any
$\beta\in C_{\sigma}^{K^\sigma}$ and every $x\in X_{\max}^{\beta}$ there is an open neighbourhood $U \subset  X^{\beta}_{max} $ of $x$ and a $G^{\beta}$-representation $F$, with all weights of $\beta$ being strictly negative and a $G^{\beta_{-}}$-equivariant biholomorhic map
	\[
		\Phi:   (G^{\beta_{-}} \times^{G^{\beta}}  F) \times U \to  (p^{\beta_-})^{-1}(U)
	\]
	such that $p^{\beta_-}\circ \Phi=p_\Omega$, where $p_\Omega$ is the projection onto $\Omega$.
\end{theorem*}

In contrast to the algebraic local structure theorem of Brion-Luna-Vust in \cite{BLV} our
result does not imply the local structure theorem for all compact $G$-orbits in  $X$. Nevertheless we show
in section \ref{section:projective} that the result in \cite{BLV} is implied by the above theorem.

\section{Basic properties of the momentum map}
\label{section:basic}

Let $Z$ be a complex manifold with a holomorphic action of the complex
reductive group $U^\C$, where $U^\C$ is the complexification of the
compact Lie group $U$.  The corresponding Cartan involution is denoted by $\Theta$.
We assume that $Z$ admits a smooth
$U$-invariant K\"ahler structure and a $U$-equivariant moment mapping
$\mu\colon Z\to\lieu^*$, where $\lieu$ is the Lie algebra of $U$ and
$\lieu^*$ its dual. As in the introduction we choose an invariant inner product on $\lieu$ and use it
to identify $\lieu^*$ with $\lieu$. In order to simplify the notations we will identify $\lieu$ as a $U$-representation isometrically with $\im\lieu$ by
multiplying  with $\im$.

We also  assume that $G\subset U\c$ is a closed
subgroup such that the Cartan decomposition $U\c=U \exp(i\lieu)\simeq
U\times i\lieu$ induces a Cartan decomposition $G=K\exp\liep\simeq
K\times \liep$ where $K=U\cap G$ and $\liep\subset i\lieu$ is an $(\Ad
K)$-stable linear subspace.  In this case we also say that $G$ is compatible with $\Theta$ or
compatible with the Cartan decomposition of $U^\C$. We have the subspace $\liep\subset\im\lieu$
and by restriction an induced restricted momentum mapping $\mup\colon
Z\to \liep$ which is the composition of the orthogonal projection of $\im\lieu$ onto $\liep$ with $\mu$.

 For $\beta\in\liep$ and $x\in Z$ let
$\mu_\liep^\beta(x):=<\mu(x),\beta>$. With our convention the momentum map condition means that
\(\grad\mup^\beta=\beta_Z \) where $\beta_Z$ is the vector field on
$Z$ corresponding to $\beta$ and $\grad$ is computed with respect
to the Riemannian metric induced by the K\"ahler structure.  We
call $\mup$ the $G$-gradient map associated with $\mu$ or the restricted momentum map. Note that the
condition on $\mup$ just means that for every $\xi\in \liep$ the vector field $(\im\xi)_Z$ is a Hamiltonian
vector field on $Z$ with Hamiltonian function $\mup^\xi$.

For a   $G$-stable subset $X$ of
$Z$ we may consider $\mu_\liep$ as a $K$-equivariant mapping $\mup\colon X\to \liep$
such that
\[
\grad\mup^\beta=\beta_X
\]
holds orbitwise. The gradient can be computed on each $G$-orbit with respect to the induced
Riemannian metric. The vector field $\beta_X$ is given by the flow $(t,x)\mapsto \exp t\beta\cdot x$.
Since $X$ is $G$-stable we have $\beta_Z(z)=\beta_X(z)$ for $z\in X$ and the map $\mup$ on $X$ is the restriction of
the map $\mup\colon Z\to \liep$ . We are interested mostly in the case where $X$ is a compact $G$-invariant subset
of $Z$.

\begin{remark}\label{remark:complexcase}
If $X$ in the above setup is smooth, then we refer to this as the smooth case.
 If $X$ is a smooth complex submanifold of $Z$ and $G$ a complex reductive subgroup
 of $U^\C$ which is compatible with the Cartan decomposition of $U^\C$, the  K\"ahler form $\omega$ on $Z$ as well as the momentum map $\mup$ restricts to $X$  and we have  $\liep=\im\liek$. We refer to this setup as the smooth complex case. This case has to be distinguished from the case where $X$ is
 smooth and complex and  $G$ is a non-complex $\Theta$-compatible Lie subgroup of $U^\C$.
\end{remark}

 We collect a few elementary properties of the restricted momentum
map $\mup\colon X\to \liep$. Details can be found in \cite{HS}.
For a subspace $\lie m$ of $\lieg$  and $x\in X$ we set $\lie m\cdot x=\{\xi_X(x): \xi\in\lie m\}$, $\liem_x:=\{\xi\in \liem: \xi_X(x)=0\}$. If  $x$ is a smooth point of $X$ and $V$ is a subspace
of $T_x X$ let $V^\perp$ denote the perpendicular subset with  respect to
the Riemannian metric $x\mapsto (\ ,\ )_x$ on the smooth part of $X$. Note that any point $x$ is a smooth
point in its $G$-orbit.

\begin{lemma}
Let $\liem\subset\liep$ be a subspace.  Then
$\ker d\mu_\liem(x)=(\liem\cdot x)^\perp$ on any smooth $G$-invariant submanifold $M$ of $Z$.
\end{lemma}
\proof This follows from $\grad\mu^\beta=\beta_X$ for $\beta\in \liep$.
\qed

\begin{lemma} \label{lemma:monotony}
 For all $\beta\in \liep$ and $x\in X$ and the curve $\gamma\colon \R\to X$, $\gamma(t)=\exp t\beta\cdot x$ we either have
 \begin{itemize}
  \item[i)] $x\in X^\beta$ or
  \item[ii)] the function $t\mapsto \mup^\beta(\gamma(t)$ is strictly increasing.
 \end{itemize}
\end{lemma}
\begin{proof}
This follows from
  \[
  \frac d{dt}(\mu^\beta\circ\gamma)(t)=\left(\beta_X(\gamma(t)),\beta_X(\gamma(t))\right)_{\gamma(t)}.
  \]
\end{proof}
For $\beta \in \liep$ let $\Mup(\beta):=\mu_{\mathfrak{p}}^{-1}(\beta)$ and
since $\beta=0$ plays a prominent role we set $\mathcal
M_{\mathfrak{p}}:=\mathcal M_{\mathfrak{p}}(0)$. Using
Lemma~\ref{lemma:monotony} and the $K$-equivariance of $\mup$, one can
prove the following

\begin{corollary}
\label{corollary:isotropycompatibel}
Let $\beta$ be a $K$-fixed point in $\liep$ and $x\in\Mup(\beta)$. Then
\begin{enumerate}
\item $G\cdot x\cap\Mup(\beta)=K\cdot x$.
\item The $G$-isotropy group
$G_x$ of $x$ is compatible with the Cartan decomposition of
$U^\C$. The decomposition is given by $G_x=K_x\exp(\liep_x)$.
\end{enumerate}
\end{corollary}
\qed

\begin{remark} If $\mup$ is a restricted momentum map and  $\alpha_0\in \liep$ is a $K$-fixed point,
 then $\tilde\mup\colon Z\to \liep$, $\tilde\mu(x)=\mup(x)+\alpha_0$ is also a restricted momentum map for the $G$-action on $Z$
 (see Remarks \ref{remark:slice-points} (\ref{item:remark:slice-points2})). If
 we assume $Z$ to be connected then  up to a additive $K$-invariant constant the restricted momentum map
 $\mup$ is uniquely defined by the Kähler structure on $Z$.
 \end{remark}

In section \ref{section:proof} we need the following.

\begin{proposition}
 \label{proposition:constant} Let $X$ be a $G$-invariant subset of $Z$ and  assume that $\mup$ is
 constant on $X$. Then for every $x\in X$ the Lie algebra $\lieg_x$ of the isotropy group $G_x$
 contains the ideal $\liep+[\liep,\liep]$ of $\lieg$.
\end{proposition}
\begin{proof} It is sufficient to show the statement orbitwise. Hence we may assume that $X$ is smooth.
For all $x\in X$ we have  $\ker d\mup(x)=(\liep\cdot x)^\perp$. Since $\mup$ is constant this shows that $\liep \cdot x=\{0\}$ holds. The isotropy Lie algebra $\lieg_x$ then contains the ideal $[\liep, \liep]+\liep$.
\end{proof}

The Lie algebra $\lieg$ of a $\Theta$-compatible subgroup $G$ is real reductive. We have a direct sum decomposition
$\lieg=\lie z(\lieg)\oplus \lieg_1\oplus\ldots \lieg_m$ where $\lie z (\lieg)$ denotes the center of $\lieg$
and the $\lieg_j$ are the simple ideals of the semisimple part of $\lieg$. Since Cartan involutions of semisimple
Lie algebras are unique up to conjugation, we have Cartan decompositions of the Lie algebras
$\lieg_j=\liek_j\oplus \liep_j$ and $\liek=\lie z_\liek \oplus \liek_1\oplus\ldots\oplus \liek_m$ and $\liep=\lie z_\liep\oplus \liep_1\oplus\ldots\oplus \liep_m$ holds where $\lie z_\liek=\lie z(\lieg)\cap \liek$ and $\lie z_\liep=\lie z(\lieg)\cap \liep$. The ideal $\lieg_j$ is said to be compact if $\lieg_j=\liek_j$ and we may assume that
$\lieg_j=\liek_j$ for $j=1,\cdots,l$. The connected Lie subgroup $G_{nc}$ of $G$ with Lie algebra
$\lieg_{nc}=\lie z_\liep\oplus \lieg_{l+1}\oplus \ldots\oplus \lieg_m$ is called the non compact factor of $G$.
The subgroup $G_{nc}$ of $G$ is a normal $\Theta$-compatible subgroup of $G$ and $G/G_{nc}$ is a compact Lie group.

\begin{corollary}
\label{corollary:noncompactfactors}
Assume that  $\mup\colon X\to \liep$ is constant and let $I:=\bigcap_{x\in X}G_x$ denote the ineffectivity of the $G$-action on $X$. Then
\begin{enumerate}
 \item $I$ is an normal subgroup of $G$ which is compatible with the Cartan decomposition,
 \label{item:corollary:noncompactfactors1}
 \item
 \label{item:corollary:noncompactfactors2}
 $I$ contains the non-compact factor of $G$.
 \item \label{item:corollary:noncompactfactors3}
 $G/I$ is a a compact Lie group. In particular every $G$-orbit is compact.
\end{enumerate}
\end{corollary}
\begin{proof}
Intersections of $\Theta$-compatible subgroups of $G$ are $\Theta$-compatible. This shows (\ref{item:corollary:noncompactfactors1}).
For (\ref{item:corollary:noncompactfactors2}) we have to show that the Lie algebra $\lie i$ of $I$ contains
$\liep$. Since $\mup$ is constant, we have  $T_x Y=\ker d\mup(x)=\liep\cdot x ^\perp$ on  $Y:=G\cdot x$. This shows $\liep\cdot x=0$ for
all $x\in X$ and implies $\liep\subset \lie i$.
\end{proof}

\section{The Slice Theorem}
\label{section:slice}

We will use the Slice Theorem of Luna type for Hamiltonian actions several times. It is valid for the action of a $\Theta$-compatible subgroup $G$ of $U^\C$ on $Z$ or slightly more generally for the $G$-action on a $G$-invariant submanifold $X$  of $Z$.  The precise statement is as follows. Recall that the points in a twisted product
$\twist{G}{H}{Y}$ are the $H$-orbits in $G\times Y$ where $H$-acts  by $(h,g,y)\mapsto (gh\inv,hy)$. In the following
we set $[g,y]:=H\cdot (g,y)\in \twist{G}{H}{Y}$ and identify $Y$ with the subset $\{[e,y]: y\in Y\}$
of $\twist{G}{H}{Y}$.

\begin{theorem} \label{theorem:slice}
(Slice Theorem) Let $X$ be a $G$-invariant submanifold of $Z$ and $x\in \Mup(\beta_0)$. Then
\begin{enumerate}
 \item $G_x=K_x\exp\liep_x$ is a $\Theta$-compatible subgroup of $G$ (see Corollary \ref{corollary:isotropycompatibel})
 \item For any $G_x$-equivariant splitting $T_xX=\lieg\cdot x\oplus W$ of the $G_x$-representation $T_xX$ there exists
 an open $G_x$-invariant neighborhood $U$ of $0$ in $W$,
 an open $G$-invariant neighborhood $\Omega$ of $x$ in $X$ and a $G$-equivariant
 diffeomorphism $\Psi\colon \twist{G}{G_x}{U}\to \Omega$ such that $\Psi([e,0])=x$
 \item In the complex case  the splitting can be chosen to be $\C$-linear and the map $\Psi$ is
 then a $G$-equivariant biholomorphism.
\end{enumerate}
\end{theorem}

\proof This shown in \cite{HS}.
\qed

\begin{remarks}
\label{remark:slice-points}
\begin{enumerate}

\item
\label{item:remark:slice-points1}
 The $G_x$-representation $T_xX$ is completely reducible for $x\in\Mup(\beta_0)$.This follows
 from the fact that every $\alpha\in \liep_x$ is represented
 on $W=T_xX$ as the  self-adjoint operator with respect to the $K_x$-invariant inner product $(\ ,\ )_x$
 on $T_xX$ (see \cite{HS}).

 \item
 \label{item:remark:slice-points2}
 The proof of the Slice Theorem can be reduced to the case where $\beta_0=0$. This is done using the shifting method
 which is also available for the $G$-action on $X$ and is as follows.
 For any $\beta_0=\mup(x)$ we replace
 $Z$ bei $Z\times O$, where $O=U\cdot \beta_0$. The manifold $O$ is a coadjoint orbit and it is a generalized $U^\C$-homogeneous flag manifold. The action of
 $U$ on $O$  is Hamiltonian with momentum map
 $\mu_O(y)=-y$. Since $O$ is a compact  K\"ahler manifold, the $U$-action extends to a holomorphic $U^\C$-action.
 In particular the group $G$ acts on $O$ by holomorphic transformations and there is exactly one closed
 $G$-orbit which is a $K$-orbit (see \cite{W}) in $O$. It turns out that $K\cdot \beta_0$ is the unique closed $G$-orbit
 in  $O$. On $Z\times O$ we have the product K\"ahler form with product momentum map $\mu+\mu_O$. The zero fibre contains
 $(x,\beta_0)$. In the case where $\beta_0$ is a $K$-fixed point, we can replace $Z$ be $Z\times O$ and identify $X$
 with the submanifold $X\times \{\beta_0\}$ of $Z\times O$. Then we can use the Slice Theorem for the zero fibre of the new
 momentum mapping.

 \item
  \label{item:remark:slice-points3}
 If $x\in X$ is a $G$-fixed point, then $\beta_0:=\mup(x)$ is a $K$-fixed point. The Slice Theorem is in this case a linearization theorem  for the $G$-action in a $G$-invariant neighborhood of $x$. It implies that the set of
 $G$-fixed points $X^G$ is a smooth submanifold of $X$.

 \item
 \label{item:remark:slice-points4}
 If in the complex case $x$ is a $K$-fixed point then it is $G=K^\C$-fixed. In this case the $G$-action can be holomorphically
 linearized.

 \item
 \label{item:remark:slice-points5}
 For a commutative $\Theta$-compatible subgroup of $G$ the Slice Theorem holds
 at every $x\in X$. In particular it holds for every subgroup $\Gamma_\beta:=\exp (\R\beta)$ where $\beta\in\liep$.

\item
\label{item:remark:slice-points6}
If $Y$ is a $G$-invariant subset of $Z$ and $x\in \Mup(\beta_0)\cap Y$ then the Slice theorem for $Z$ gives by pulling back a Slice Theorem for $Y$. The map $\Psi\colon \twist{G}{G_x}{U_Y}\to \Omega\cap Y$ is a $G$-equivariant homeomorphism, where $U$ is identified with $\{[e,u]\in \twist{G}{G_x}{U}: u\in U\}$
and $U_Y:=\Psi\inv(Y\cap \Psi(U))$.
 \end{enumerate}
\end{remarks}

\section{Parabolic subgroups}
\label{section:parabolic}
Let $G$ be a $\Theta$-compatible subgroup of $U^\C$ and $G=K\exp\liep$ the Cartan
decomposition of $G$.

Let $\beta\in\liep$. We have the vector field $\beta_G$ whose
one-parameter subgroup is given by
$(t,y)\mapsto
\exp(t\beta)y\exp(-t\beta)$. Then
$$G^\beta=\{y\in G: \beta_G(y)=0\}=\{y\in G: \exp(t\beta)y\exp(-t\beta)=y \text{ for
  all $t\in \R$}\}
  $$
  is the centralizer of $\beta$ in $G$.
We have the  parabolic subgroup
\[
G^{\beta_+}:=\{y\in G:
\lim_{t\to-\infty}\exp(t\beta)y\exp(-t\beta)\text{ exists}\}
\]
with unipotent radical
\[
R^{\beta_+}:=\{y\in G:
\lim_{t\to-\infty}\exp(t\beta)y\exp(-t\beta)=e\}.
\]
Then $G^{\beta_+}$ is the semi-direct product of $G^\beta$ with
$R^{\beta_+}$ and we have the projection
  $\pi^{\beta_+}\colon G^{\beta_+}\to G^\beta$,
$\pi^{\beta_+}(y):=\lim_{t\to-\infty}\exp(t\beta)y\exp(-t\beta)$.
The opposite of $G^{\beta_+}$ is the parabolic subgroup
$G^{\beta_-}=\Theta(G^{\beta_+})=G^\beta\cdot R^{\beta_-}$, where
\[
R^{\beta_-}:=\{y\in G:
\lim_{t\to+\infty}\exp(t\beta)y\exp(-t\beta)=e\}.
\]
We introduce subsets  of a $G$-invariant closed subset $X$ of $Z$  analogous to $G^\beta$,
$G^{\beta_+}$ and $G^{\beta_-}$. For $\beta\in\liep$ we have the
vector field $\beta_X$ on $X$ along $G$-orbits whose one-parameter subgroup is
given by $(t,y)\mapsto \exp(t\beta)\cdot y$.  We have the set of $\Gamma_\beta:=\exp\R\beta$ fixed points $X^\beta=\{y\in X:\beta_X(y)=0\}=X\cap Z^\beta$. For a real number $r$ let
$X{(\beta,r)}:=\{x\in X:  \overline{\exp\R\beta\cdot x}\cap(\mu^\beta)^{-1}(r)\not=\emptyset\}$.
Since $X$ is closed in $Z$ these subsets are also given by intersecting the analogous subsets of $Z$ with $X$.

The set $X{(\beta,r)}$ can be viewed as a set of semistable points in $X$ for the
$\Gamma_\beta=\exp\R\beta$-action on $X$ in the sense of \cite{HSS}. Note that $\Gamma_\beta$ is a $\Theta$ compatible subgroup of $G$. In particular $X{(\beta,r)}$
is an open subset of $X$ and there is a topological Hilbert quotient $\quot{X{(\beta,r)}}{\Gamma_\beta}$ (\cite{HS}).

 Since $\grad \mu^\beta=\beta_Z(x)$ the set
\[
 X_r^\beta:= X{(\beta,r)}\cap Z^\beta
\]
is open and closed in $X^\beta$. We have  $G^\beta=K^\beta\exp\liep^\beta$ and $\mup^\beta$
is $K^\beta$-invariant. This implies that $G^\beta$ stabilizes $X^\beta_r$ and therefore also the subset
\[
X^{\beta_-}_r:=\{y\in X{(\beta,r)}: \lim_{t\to + \infty}\exp t\beta\cdot y\text{
  exists in } X{(\beta,r)} \}.
\]
In particular $X^{\beta_-}_r$ ist a $G^{\beta_-}$-stable set
and we also have
the $G^{\beta_+}$-stable set
\[
X^{\beta_+}_r:=\{y\in X{(\beta,r)}: \lim_{t\to - \infty}\exp t\beta\cdot y\text{
  exists in } X{(\beta,r)} \}.
\]

The map $p^{\beta_+}\colon X^{\beta_+}_r\to X_r^\beta$,
$p^{\beta_+}(y)=\lim_{t\to -\infty}\exp t\beta\cdot y$ is well
defined, $G^\beta$-equivariant, surjective and its fibers are
$R^{\beta_+}$-stable.
Similarly we have the map $p^{\beta_-}\colon X^{\beta_-}_r\to X_r^\beta$,
$p^{\beta_-}(y)=\lim_{t\to +\infty}\exp t\beta\cdot y$ which is a well
defined, $G^\beta$-equivariant, surjective map whose fibers are
$R^{\beta_-}$-stable. Of course for most $r$ these sets are empty, but if not then
they are locally closed in  $X$.  In the smooth case they are locally closed submanifolds of $X$.

\begin{remark}
 One  can define the above subsets for any $G$-invariant subset $X$ of $Z$. In this case we have
 for $X^\beta_r=Z_r^\beta\cap X$ and $X^{\beta_-}_r\subset Z^{\beta_-}_r\cap X$. If $X$ is closed
 in $Z$, then $X^{\beta_-}_r = Z^{\beta_-}_r\cap X$ holds and similar for the $+$-case.
\end{remark}

The Slice Theorem applied to the $\Gamma_\beta$ action at $x\in X^\beta$
shows the following.

\begin{proposition}
 \label{proposition:slicetheoremforgamma} Let $X$ be a smooth $G$-invariant submanifold of $Z$ and let $x\in X^\beta$.
 \begin{enumerate}
 \label{proposition:slicetheoremforgamma1}
  \item If $x\in X^\beta$, then we have $T_xX=W^{\beta_-}\oplus W^\beta\oplus W^{\beta_+}$ where $W^{\beta_-}=T_x(p^{\beta_-})\inv(x)$
 is the sum of strictly negative eigenspaces,  $W{^\beta_+}=T_x((p^{\beta_+})\inv(x))$
 is the sum of strictly positive eigenspaces and $W^\beta$ is the sum of zero eigenspaces of the linearized vector field $d\beta_X(x)$.
 The point $x$ has an open $\Gamma_\beta$-stable neighbourhood which can be equivariantly and smoothly identified with an open
 invariant neighbourhood of zero in $T_xX$.
 \item
 \label{proposition:slicetheoremforgamma2} In the complex case let $\Gamma_\beta^c$ denote the complex Zariski closure of
 $\Gamma_\beta$. Then $x$ is fixed by the algebraic torus $\Gamma_\beta^c$ and the splitting is $\Gamma_\beta^c$-invariant.
 \end{enumerate}
 \end{proposition}
\qed

\begin{remark} The proposition applies in particular to $Z$. Note that if $x\not\in Z^\beta$, then Lemma \ref{lemma:monotony} implies that $\Gamma_\beta$ acts freely in a neighbourhood of $x$ and the
 Slice Theorem gives in this case a neighbourhood isomorphic with a product $\Gamma_\beta\times S$.
 Also note that in the complex case the $\Gamma_\beta^c$-isotropy group at $x$  need not be trivial.
\end{remark}

\section{Extreme points}
In this section  $X$ denotes a compact $G$-invariant subset of $Z$. As before the group $G$ is assumed to  be compatible with $\Theta$.

Let $\sigma$ be an extreme point of the convex envelope $E$ of $\mup(X)$  and $K^\sigma$ the centralizer of $\sigma$ in $K$.

\begin{remark}
 \label{remark:polytop}
 Let $\liea$ be a maximal subalgebra of $\lieg$ which is contained in $\liep$. The convex envelope $A=\exp \liea$ is a $\Theta$-compatible subgroup of $G$. The convex envelope $P$
 of $\mua(X)$ is simpler to understand than $E$ and encodes all relevant information about
 the extreme points of $E$. The convex body $P$ is the convex envelope of
 the finite subset $\mu_\liea(X^A)$ and is therefore a convex polytope. Every extreme point $\sigma$ of $E$ lies in the $K$-orbit  of an extreme point of $P$ which up to the action of the Weyl group on $\liea$ is uniquely defined by $\sigma$ (see \cite{BGH2} for details).

 In the complex connected case it is well known that $\mua(X)$ is convex (\cite{GS,atiyah}). In the general compact smooth case it is only known that $\mua(X)$ is a finite union of convex polytopes. Convexity is an open question.
 (see \cite{HSch}).
\end{remark}

Every face of $E$ is exposed (\cite{BGH2}). In particular,  the set  $C_\sigma$ of $\beta\in\liep$ which expose $\sigma$ is
non empty. By definition $\beta\in \liep$ exposes $\sigma$ if $E\cap H^+_{\langle  \beta,\sigma\rangle}(\beta)=\{\sigma\}$ where for any $r\in\R$ we set
$H_r^+(\beta):=\{\alpha\in \liep: \langle \beta, \alpha\rangle\ge r \}$. The set
$C_\sigma$ is a $K^\sigma$-invariant cone in $\liep$ and since $K^\sigma$ is compact and acts linearly on $\liep$ the set
$C_\sigma^{K^\sigma}$ of $K^\sigma$-fixed points in $C_\sigma$ is non empty.

\begin{proposition}
 \label{proposition:extreme}
 Let $\sigma$ be an extreme point of $E$.
 \begin{enumerate}
  \item
  \label{item:extreme1}
  For every $\beta\in C_\sigma$ we have $\mup\inv (\sigma)=X_{\max}^\beta:=\{x\in X:\mup^\beta(x)=\max\mup^\beta(X)\}$.
  \item
  \label{item:extreme2}
  $X_{\max}^\beta$ is open and closed in  $X^\beta$ and $G^\beta$-stable.
    \item
  \label{item:extreme3}
   For $\beta\in C_\sigma$ every $G^\beta$-orbit in $X^\beta_{\max}$ is compact and $X^\beta_{\max}\subset X^{\liep^\beta}$.
    \item
  \label{item:extreme4}
  For $\beta\in C_\sigma$ and all $x\in X^\beta_{\max}$ we have $\lie r^{\beta_+}\cdot x=\{0\}$
   \item
  \label{item:extreme5}
   For $\beta\in C_\sigma^{K^\sigma}$ we have $K^\sigma=K^\beta$.
  \item
  \label{item:extreme6}
  For $\beta\in C_\sigma^{K^\sigma}$ and $x\in X^\beta_{\max}$  we have $\xi_X(x)\not=0$ for all non zero
  $\xi\in \lie r^{\beta_-}$.
 \end{enumerate}
\end{proposition}
\proof
Part (\ref{item:extreme1}) follows directly from the definitions.
Since $\mu^\beta$ is locally constant on $X^\beta$ we have the first part of (\ref{item:extreme2}).
The function  $\mu^\beta$ is $K^\beta$-invariant and $X^\beta$ is $G^\beta$-stable. This shows that  $X^\beta_{\max}$
is $G^\beta$-stable.

Since $\mup\vert X^\beta_{\max}$ is constant by (\ref{item:extreme1}) the restricted momentum map
for the $G^\beta$-action on $X^\beta_{\max}$ is constant. Proposition \ref{proposition:constant} now implies
(\ref{item:extreme3}).

In (\ref{item:extreme4}) we apply Proposition \ref{proposition:slicetheoremforgamma} in the case where $X=G\cdot x$. Using the notations there we have $W^{\beta_+}=\{0\}$ for every $x\in X^\beta_{\max}$. This implies
$\lie r^{\beta_+}\cdot x=\{0\}$.

If $\beta$ is $K^\sigma$-fixed, then we have $K^\sigma\subset K^\beta$. Since $\mup$ is $K$-equivariant and $K^\sigma$ is the
$K$-isotropy group at $\sigma$ we also have $K^\beta\subset K^\sigma$. This shows (\ref{item:extreme5}).

Let $\xi\in \lie r^{\beta_-}$ be such that $\xi_X(x)=0$. Since $r^{\beta_+}\cdot x=\{0\}$ and
$\theta(\xi)\in \lie r^{\beta_+}$ we get  that $(\xi+\theta(\xi))_X(x)=0$. Then we have
that $\xi +\theta(\xi)\in\liek_x\subset \lie k^\sigma=\liek^\beta\subset \lieg^\beta\subset \lie r^{\beta_-}+\lieg^\beta+\lie r^{\beta+}$.
Since the last sum is direct we obtain $\xi=0$.
\qed

\begin{corollary}
\label{corollary:proposition:extreme}
Let  $\beta\in C_\sigma$.
\begin{enumerate}

  \item For all $x\in X^{\beta}_{\max}$ the orbit $G\cdot x$ is a $K$-orbit.
 \item In the smooth case $X^{\beta_-}_{\max}$ is open in $X$.
 \item If $X$ is connected and we are in the complex case then $X^{\beta_-}_{\max}$ is dense in $X$.
  \end{enumerate}
\end{corollary}

\begin{corollary}
 \label{corollary:extreme:complexcase}
 Assume that $G$ is connected and that we are in the complex case. Then we have  $X^\beta_{\max}\subset X^{G^\beta}$ for $\beta\in C_\sigma$ and $G^\beta=G^\sigma$ for
 $\beta\in C_\sigma^{K^\sigma}$.
\end{corollary}

\proof
Since $G^\beta$ is connected for a connected complex reductive group $G$ and $\liep^\beta=\im\liek^\beta$
we have $X^{\liep^\beta}=X^{\lieg^\beta}=X^{G^\beta}$.
\qed

\begin{remark}
 It is shown in \cite{BGH2} that in the case of $\beta\in C_\sigma^{K^\sigma}$ the bundle  $p^{\beta_-}\colon X^{\beta_-}\to X^\beta_{\max}$ does not depend on the choice of $\beta\in C^{K^\sigma}_{\beta}$.
\end{remark}

\section{Proof of the structure theorem}
\label{section:proof}
In this section we proof a local structure theorem at any point $x\in \mup\inv(\sigma)$, where $\sigma$ is an extreme point of the convex envelope $E$ of $\mup(X)$ and $X$ is a $G$-invariant compact submanifold  of $Z$.

We fix a $\beta\in C_\sigma^{K^\sigma}$ (see Proposition \ref{proposition:extreme}) and  let $I^\beta$ be the
ineffectivity of the $G^\beta$-action  on $X^\beta_{\max}$. Then  $I^\beta$ is a $\Theta$-compatible normal subgroup of $G^\beta$
(see Corollary \ref{corollary:noncompactfactors}). Its
Lie algebra contains the ideal $[\liep^\beta,\liep^\beta]+\liep^\beta$ and  $I^\beta$ is a $\Theta$ compatible normal subgroup of $G^\beta$ which contains the non-compact factor of $G^\beta$. The group $R^{\beta_-}\rtimes I^\beta$ acts fiberwise on the bundle $p^{\beta_-}\colon X^{\beta_-}_{\max}\to X^\beta_{\max}$.

\begin{remark}
 \label{remark:complexcase2}
 \begin{enumerate}
  \item If $G$ is connected and  we are in  the complex case, then $I^\beta=G^\beta$ (Corollary \ref{corollary:extreme:complexcase}) and $R^{\beta_-}\rtimes G^\beta=G^{\beta_-}$.
  \item If we have  $(p^{\beta_-})\inv(\sigma)=\{x\}$, then $I^\beta=G^\beta$.
 \end{enumerate}
\end{remark}

Since $I^\beta$ is $\Theta$-compatible and fixes  $X^\beta_{\max}$ pointwise we may apply the Slice Theorem at $x\in X^\beta_{\max}$. For the $I^\beta$-representation $T_xX$ we have an $I^\beta$-equivariant splitting $T_xX=W^{\beta_-}\oplus W^\beta$
where $W^\beta$ is the
$I^\beta$-submodule which is pointwise fixed and $W^{\beta-}$ is the $I^\beta$ submodule such that
the linearized vector field $d\beta_X(x)$ acts by strictly negative weights. Note that $W^{\beta_+}=\{0\}$, since $\mup^\beta$ is maximal in $x$. This implies
that the only $\Gamma_\beta:=\exp\R\beta$ stable  neighborhood of zero in $W^{\beta_-}$ is the entire space $W^{\beta_-}$.
The Slice Theorem is given by an open  neighborhood $U$ of $0$ in $W^\beta$, an $I^\beta$-invariant open neighborhood $\Omega$ of $x$ in X and an $I^\beta$-equivariant diffeomorphism
$ \Psi\colon W^{\beta_-}\times U \to \Omega$.   The diffeomorphism $\Psi$ preserves the fibers of the fibration
$\pi_U\colon W^{\beta_-}\times U\to U$ and $p^{\beta_-}\colon X^{\beta_-}_{\max}\to X^\beta_{\max}$. In particular
$\Omega=(p^{\beta_-})\inv(V)$ for some open neighborhood $V$ of $x$ in $X^\beta_{\max}$ and $\Psi$ maps $U$ diffeomorphically
onto $V$. From this we get a local  $I^\beta$-equivariant trivialization  $V\times W^{\beta_-}\to (p^{\beta_-})\inv(V)$ of the fibration.

\begin{lemma}
 \label{lemma:unipotentliniarization}
 The tangent space $\lie r^{\beta_-}\cdot x$ of the $R^{\beta_-}$-orbit through $x$ is an
 $I^\beta$-submodule of $W^{\beta_-}$. The differential $\varphi_x\colon \lie r^{\beta_-}\to
 \lie r^{\beta_-}\cdot x$ of the orbit map $\Phi_x\colon R^{\beta_-}\to X$, $g\mapsto g\cdot x$, is an $I^\beta$-equivariant isomorphism.
\end{lemma}

\proof
For $h\in I^\beta$ and $g\in R^{\beta_-}$ we have $hgh\inv\cdot x=h\cdot (g\cdot x)$.
This shows that $\varphi_x$ is an $I^\beta$-equivariant surjective map.
Proposition \ref{proposition:extreme} (\ref{item:extreme6}) implies injectivity of $\varphi_x$.
Since $W^{\beta_-}$ is the tangent space of the fiber of $(p^{\beta_-})\inv(x)$ which contains $R^{\beta_-}\cdot x$ we also have $\lie r^{\beta_-}\cdot x\subset W^{\beta_-}$.
\qed

\begin{remark}
 Since $I^\beta$ is $\Theta$-compatible, the representation $T_xX$ is completely reducible.
\end{remark}

\begin{theorem}
\label{theorem:structuretheorem} Let $X$ be a compact $G$-invariant submanifold of $Z$, $\sigma$ an extreme
point of the convex envelope of $\mup(X)$ in $\liep$, $\beta\in C_\sigma^{K^\sigma}$ and $x\in (p^{\beta_-})\inv(\sigma)$.

Let $W^{\beta_-}=\lie r^{\beta_-}\cdot x\oplus F$ be an $I^\beta$-equivariant splitting and $\Psi\colon W^{\beta_-}\times U\to \Omega$ the diffeomorphism
obtained by the Slice Theorem. Then we can choose the open neighbourhood $U$ of $0\in W^\beta$ such that the map
\[
 \Phi\colon R^{\beta_-}\times F\times U\to (p^{\beta_-})\inv (V),\ (g,v,u)\mapsto g\cdot \Psi(v,u)
\]
 is a diffeomorphism.

 The map is $R^{\beta_-}\rtimes I^\beta$-equivariant where the action on
 $R^{\beta_-}\times F\times U$ is given by $(r,h),(g,v,u)\mapsto (rgh\inv,h\cdot v, u)$.
\end{theorem}

\proof
It follows from the definition that the map $\Psi$ is equivariant.

Let $\gamma\colon\R\to I^\beta$ be the one-parameter group $\gamma(t)=\exp t\beta$. For every
$(g,v,u)\in R^\beta\times F\times U$ we have
\[
 \lim_{t\to +\infty}\gamma(t)\cdot (g,v,u)
=   \lim_{t\to +\infty} (\gamma(t)g\gamma(t)\inv,\gamma(t)\cdot v, u)=(e,0,u).
\]
and
\[
  \lim_{t\to +\infty}\gamma(t)\cdot \Phi(g,v,u)
=   \lim_{t\to +\infty} (\gamma(t)g\gamma(t)\inv\cdot \Psi(\gamma(t)\cdot v,u))=\Psi(0,u).
\]

The differential  of $\Phi$ at the point $(e,0,0)$ is an isomorphism and since $\Phi$ is $\gamma(\R)$-equivariant it
 is a local diffeomorphism in a $\gamma(\R)$-invariant neighbourhood of $(e,0,0)$. Every open $\gamma(\R)$-invariant neighborhood
of $(e,0,0)$ contains an open product neighbourhood $\Omega:=N\times U\times \tilde U$ of $(e,0,0)$ such that
$\Phi\vert\Omega$ is a diffeomorphism. Since for $x,y\in \Gamma_\beta\cdot(N\times U\times \tilde U)= R^{\beta_-}\times U\times F$
there is a $t\in \R$ such that $\gamma(t)\cdot x$, $\gamma(t)\cdot y\in \Omega$ the map $\Phi$ is a diffeomorphism
on $R^{\beta_-}\times U\times F$ with image $(p^{\beta_-})\inv(V)$ where $V=\Phi(\{e\}\times U\times \{0\})=\Psi(U\times \{0\} )$.
\qed

\begin{remark}
 \label{remark:theorem:structuretheorem:complexcase}
\begin{enumerate}
 \item In the complex case and for a connected $G$ we have $I^\beta=G^\beta$. The map $\Phi$ is biholomorphic and the inclusion
 $R^{\beta_-}\times F\subset G^{\beta_-}\times F$ induces a $G^{\beta_-}$-equivariant biholomorphic map
 \[R^{\beta_-}\times F\times U\to  \twist{G^{\beta_-}}{G^\beta}{F}\times U.\]
 \item In the real case the inclusion $R^{\beta_-}\times F\times U\subset (I^{\beta_-}\times F)\times U$ induces an $I^{\beta_-}$-equivariant diffeomorphism
 \[R^{\beta_-}\times F\times U\to  (\twist{(R^{\beta_-}\rtimes I^\beta)}{I^\beta}{F})\times U.\]
    \item If we replace $X^\beta_{\max}$ by some union $C$ of its connected components, then in a  similar way as
    above we can define $I^\beta$ for this union $C$. Then Theorem \ref{theorem:structuretheorem} holds for this modified $I^\beta$ and any $x\in C$.
  \item In the complex connected case $X^\beta_{\max}$ is connected.
\end{enumerate}
\end{remark}

\begin{corollary}
 \label{corollary:isolated}
 Let a connected component of $X^\beta_{\max}$ be an isolated point $x$ and let $I^\beta=G^\beta_x$. Then
 $I^\beta$ contains the connected component of the  identity of $G^\beta$, $(p^{\beta_-})\inv (x)$ is open in $X$
 and $I^\beta$-equivariantly diffeomorphic to  $\lie r^{\beta_-}\times W^{\beta_-}$.
 If $X^\beta_{\max}=\{x\}$, then we have $I^\beta=G^\beta$
\end{corollary}

\section{Quotients} \label{section:quotients}

In this section $X$ is a $G$-invariant compact smooth submanifold of $X$, $\sigma$
is an extreme point of the convex envelope of $\mup(X)$ and $\beta\in C_\sigma^{K^\sigma}$.

As a consequence of the structure theorem we have that $q\colon X^{\beta_-}_{\max}\to X^{\beta_-}_{\max}/ R^{\beta_-}$
is an $R^{\beta_-}$-principal bundle.

\begin{corollary}
 \label{corollary:properandfree}
 The action of $R^{\beta_-}$ on $X^{\beta_-}$ is proper and free.
\end{corollary}

\proof Theorem \ref{theorem:structuretheorem} shows that the $R^{\beta_-}$ action is free. Let $(g_l,y_l)$ be a sequence in
$R^{\beta_-}\times X^{\beta_-}_{\max}$ such that $g_l\cdot y_l$ and $y_l$ converge in $X^{\beta_-}_{\max}$. Since $p^{\beta_-}$ is
$R^{\beta_-}$-invariant both sequences converges in $X^{\beta}_{\max}$. Theorem \ref{theorem:structuretheorem} now shows
that $g_l$ has a convergent subsequence.
\qed

The $G^\beta$-action on $X^{\beta_-}_{\max}$ induces an $G^\beta$-action on the quotient $X^{\beta_-}_{\max}/ R^{\beta_-}$.
 On $X^{\beta}_{\max}$ we have an
artificial $R^{\beta_-}\rtimes G^\beta$-action given by the projection $R^{\beta_-}\rtimes G^\beta\to G^\beta$.  Since $p^{\beta_-}\colon X^{\beta_-}_{\max} \to X^\beta_{\max}$ is $R^{\beta_-}\rtimes I^\beta$-invariant we  obtain
an $I^\beta$-invariant map $ p\colon X^{\beta_-}_{\max}/ R^{\beta_-}\to X^{\beta}_{\max}$ such that
\[
p\circ q=p^{\beta_-} .
\]
Theorem \ref{theorem:structuretheorem} now implies

\begin{theorem}
 $p\colon X^{\beta_-}_{\max}/ R^{\beta_-}\to X^{\beta}_{\max}$ is a topological Hilbert quotient.
\end{theorem}

\section{Compact orbits}
\label{section:compactorbits}

Let $X$ be a compact $G$-invariant subset of $Z$ and $\eta\colon Z\to \R$, $\eta(x)=\Vert \mup(x)\Vert^2$. We have $\grad\eta(x)=\beta_X(x)$ where $\beta:=\mup(x)$. Note that the flow of $\eta$ is along $G$-orbits.

\begin{proposition}
\label{proposition:extremanorm}
 Let $x\in X$ be such $\eta(x)$ is maximal and $\beta:=\mup(x)$. Then $\beta$ is an extreme point of
 the convex envelope of $\mup(X)$, and $x\in X_{\max}^\beta$.
\end{proposition}

\begin{proof}
 Let $E$ be the convex envelope of $\mup(X)$. For the point $x$ we have $\grad\eta(x)=\beta_X(x)$. Since $x$ is maximal for $\eta\vert G\cdot x$ we
 obtain $x\in X^\beta$. We also have $\mup^\beta(x)=<\mup(x),\beta>=\eta(x)$. Since $\eta$ is maximal in $x$
 we have $\eta(y)\le \eta(x)$  for every $y\in X$. This implies
 $\mup^\beta(y)=<\mup (y),\beta>\le \eta(x)=\mu^\beta(x)$ and shows $x\in X^\beta_{\max}$. Since the affine plane $H_{\Vert \beta \Vert^2}(\beta)$ through $\beta$ with normal vector $\beta$  intersects the sphere
 of radius $\Vert \beta \Vert$ around zero only in the point $\beta$ and $\mup(X)\subset H_{\mup^\beta(x)}(\beta)$
 it follows that $\beta$ is an extreme point of $E$.
\end{proof}

\begin{lemma}
\label{lemma:orbitclosure}
\begin{enumerate}
 \item Every $G$-orbit in $X$ has a compact orbit in its closure.
 \item Every compact $G$-orbit in $Z$ is a $K$-orbit.
\end{enumerate}
\end{lemma}

\begin{proof}
 For $z\in X$ let $Y:=\overline{G\cdot z}$ be the closure of the $G$-orbit through $z$. Since
 $Y$ is compact there is an extreme point $\sigma$ of the convex envelope of $\mup(Y)$. Then
 the $G$-orbit through $y\in Y^\beta_{\max}$ is a $K$-orbit (Corollary \ref{corollary:proposition:extreme}).
\end{proof}

\begin{remark}
 Even in the smooth compact complex case not every compact $G$-orbit in $X$ can be realized as an orbit
 through a point in $X^\beta_{\max}$ for some $\beta$.
 \end{remark}

The case where the compact set $X$ contains exactly one compact $G$-orbit is the simplest possible. We collect few properties which hold in this case.

 \begin{lemma}
  Assume that $X$ contains exactly one compact $G$-orbit $Y$. Then $\mup(Y)$ is a $K$-orbit in $\liep$ and this $K$-orbit coincides with the set of extreme points of the convex envelope $E$ of $\mup(X)$. If $\sigma$ is an extreme
  point then $\sigma$ exposes $\sigma$ and $R^{\beta_-}$ acts freely on $X^{\sigma_-}_{\max}$.
\end{lemma}

\begin{proof}
It follows from \ref{proposition:extremanorm} that the set of extreme points of $E$ is a $K$-orbit. The freeness
of the $R^{\beta_-}$-action follows from $\sigma\in C_\sigma^{K^\sigma}$ and \ref{proposition:extreme} (\ref{item:extreme6}).
\end{proof}

Finally we note the following.

\begin{lemma}
 Let $X$ be a compact smooth $G$-invariant submanifold of $X$ with exactly one compact orbit. Then
 $X=G\cdot X^{\beta_-}_{\max}$.
\end{lemma}

\begin{proof}
 The set $G\cdot X^{\beta_-}_{\max}$ is open in $X$ and $Y=X \setminus G\cdot X^{\beta_-}_{\max}$ is
 compact. If we assume that it is non empty then Lemma \ref{lemma:orbitclosure} gives a contradiction.
\end{proof}

\section{The projective case}
\label{section:projective}

In this section we give a proof of the result of Brion-Luna-Vust (\cite{BLV}).
Assume that $G=K^\C$ is a connected complex reductive group with maximal compact subgroup $K$. Let $V$ be a
finite dimensional
unitary representation of $K$ and set $X:=\mathbb{P}(V)$. The action of $K$ is Hamiltonian with
respect to the standard K\"ahler metric on $\mathbb P(V)$ and we have a momentum map $\mu\colon X\to \im\liek$. For a given linear subspace $W$ of $V$ we write in the
following $\mathbb P(W)$ for the image of $W\setminus\{0\}$ in $\mathbb{P}(V)$.

Let $G\cdot x $ be a compact orbit in $\mathbb{P}(V)$. Then there is an irreducible subrepresentation $W$ of
$V$ such that $G\cdot x\subset\mathbb{P}(W)$. Note that $\mathbb{P}(W)$ contains exactly one compact $G$-orbit.
Let $\beta$ be an extreme point of the convex envelope of $\mu(\mathbb{P}(W))$. Note that
$\mathbb{P}(W)^\beta_{\max}$ is an isolated fixed point of $G^\beta$ in $\mathbb{P}(W)$.

For $W$ we have the orthogonal decomposition $W=W_{\beta_0}\oplus W_{\beta_1}\ldots \oplus W_{\beta_m}$ into $\beta$ eigenspaces
with eigenvalues $\beta_j$ where $W_{\beta_0}=W^\beta$ is the line in $W$ which corresponds to $x$.

The Slice Theorem shows that $\mathbb{P}(W)^{\beta_-}_{\max}$ can be identified with a $G^\beta$-representation.
Let $w_0$  be a basis of $ W_{\beta_0}$. For $w=w_0\oplus w_1\oplus \ldots\oplus  w_m\in W\setminus\{0\}\subset
W_{\beta_0}\oplus W_{\beta_1}\ldots \oplus W_{\beta_m}$  we have
for its image $[w]$ in  $\mathbb P(W)$ that
\[
\exp t\beta\cdot [w]=[w_0\oplus e^{\beta_1-\beta_0}w_1\oplus \ldots\oplus  e^{\beta_m-\beta_0}w_m]
\]
holds for all $t\in\R$. Note that $w_0\not= 0$ is a basis of $W_{\beta_0}$ and that $\beta_j-\beta_0<0$ for $j=1,\cdots,m$. This shows that
\[
\mathbb{P}(W)^{\beta_-}_{\max}=[w_0\oplus W_{\beta_1}\ldots \oplus W_{\beta_m}]=\mathbb{P}(W)\setminus\mathbb{P}((W_{\beta_0})^{\perp_W}) .
\]
Since eigenspaces corresponding to different eigenvalues are orthogonal we have the
orthogonal decomposition
\[
W=W_{\beta_0}\oplus  \lie r^{\beta_-}\cdot w_0\oplus (\lieg\cdot w_0)^{\perp_W}
\]

The representation $F$ in Theorem \ref{theorem:structuretheorem} is now given as a subset of $\mathbb{P}(W)$ by
\[
\mathbb{P}(w_0\oplus (\lieg\cdot w_0)^\perp)=\mathbb{P}(W_{\beta_0}\oplus (\lieg\cdot w_0)^\perp) \setminus \mathbb{P}((W_{\beta_0})^\perp) .
\]

Theorem \ref{theorem:structuretheorem} shows the following.

\begin{theorem}
\label{theorem:irreducible}
The  restriction
\[
\Phi\colon R^{\beta_-}\times (\mathbb{P}(W_{\beta_0}\oplus (\lieg\cdot w_0)^{\perp_W}) \setminus \mathbb{P}((\lieg\cdot w_0)^{\perp_W}))\to \mathbb{P}(W)\setminus\mathbb{P}((W_{\beta_0})^{\perp_W})
\]
 of the action map is an isomorphism.
 \end{theorem}

 In the general case we have a orthogonal decomposition $V=W\oplus U$, $W=W_{\beta_0}\oplus (W_{\beta_0})^{\perp_W}$,
 $W_{\beta_0}^{\perp_V}=W_{\beta_0}^{\perp_W}\oplus U$ and
 $(\lieg\cdot w_0)^{\perp_V}=(\lieg \cdot w_0)^{\perp W}\oplus U$. Note that $\mathbb{P}(U)$ is a $G$-invariant subset
 of $\mathbb P(V)$. The following  is shown in \cite{BLV}.

 \begin{corollary}
  The restriction
  \[
  \Phi\colon R^{\beta_-}\times (\mathbb{P}(W_{\beta_0}\oplus (\lieg\cdot w_0)^{\perp_V}) \setminus \mathbb{P}((W_{\beta_0})^{\perp_V}))\to \mathbb{P}(V)\setminus\mathbb{P}((W_{\beta_0})^{\perp_V}
  \]
  of the action map is an isomorphism.
 \end{corollary}

\begin{proof}
Let  $z\in \mathbb{P}(W_{\beta_0}\oplus (W_{\beta_0})^{\perp_{W}}\oplus U)\setminus \mathbb{P}((W_{\beta_0})^{\perp_V})$. As above let $w_0$ be a basis of $W_{\beta_0}$. Since $z\notin \mathbb{P}((W_{\beta_0})^{\perp_V})$ we have $[z]=[ w_0+w_1+u]$ with
$w_1\in (\mathbb{C}{w_0})^{\perp_{W}}$ and $u \in U$.
Using the result in the irreducible case  there are  $r\in  R^{\beta_{-}}$, $\alpha\in \C^*$ and
$w\in (W_{\beta_0})^{\perp_{W}}$ such that $r \cdot [(\alpha w_0+w)+u]=[r\cdot (\alpha w_0+w)+r\cdot u]=[w_0+w_1+r\cdot u]$. This shows that $r\cdot [(\alpha w_0+w)+ r\inv\cdot u]=[(w_0+w_1)+ u]=z$.

For injectivity assume that $r\cdot [w_0+a_1+b_1]=[w_0+a_2+b_2]$, where
$a_1, a_2 \in (W_{\beta_0})^{\perp_{W}}$ and $b_1,b_2 \in U$. It follows that $r\cdot (w_0+a_1+b_1)
= \lambda  ( w_0+a_2+b_2) $ with $\lambda \in \mathbb{C}^{*}$. This shows  $r\cdot ( w_0+a_1)= \lambda (w_0+a_2)$ and $r\cdot b_1=  \lambda  b_2 $. Theorem \ref{theorem:irreducible}  implies
$r=e$ and we obtain $[w_0+a_1+b_1]=[\alpha w_0+a_2+b_2]$

\end{proof}

\bibliographystyle{letter}


\end{document}